\documentclass[reqno]{amsart}
\usepackage{pictexwd,dcpic}
\usepackage{slashbox}

\newtheorem{theorem}{Theorem}[section]
\newtheorem{proposition}[theorem]{Proposition}
\newtheorem{lemma}[theorem]{Lemma}
\newtheorem{corollary}[theorem]{Corollary}

\theoremstyle{definition}
\newtheorem{definition}[theorem]{Definition}

\theoremstyle{remark}
\newtheorem{remark}[theorem]{Remark}

\numberwithin{equation}{section}

\begin{document}

\title{Modified Jacobi forms of index zero}

\author{Ja Kyung Koo}
\address{Department of Mathematical Sciences, KAIST}
\curraddr{Daejeon 373-1, Korea} \email{jkkoo@math.kaist.ac.kr}
\thanks{}

\author{Dong Hwa Shin}
\address{Department of Mathematical Sciences, KAIST}
\curraddr{Daejeon 373-1, Korea} \email{shakur01@kaist.ac.kr}
\thanks{}


\keywords{Jacobi forms, Klein forms, modular forms.
\newline The first author was partially supported by Basic Science Research Program through
the NRF of Korea funded by MEST (2010-0001654).}

\begin{abstract}
By modifying a slash operator of index zero we define
\textit{modified Jacobi forms} of \textit{index zero}. Such forms
play a role of generating nearly holomorphic modular forms of
integral weight. Furthermore, by observing a relation between the
coefficients of Fourier development of a modified Jacobi form we
construct a family of finite-dimensional subspaces.
\end{abstract}

\maketitle

\section {Introduction}

Let $\mathfrak{H}$ denote the complex upper half-plane
$\{\tau\in\mathbb{C}:\mathrm{Im}(\tau)>0\}$. The letters $\tau$ and
$z$ will always stand for variables in $\mathfrak{H}$ and
$\mathbb{C}$, respectively. For fixed integers $k$ and $m$ ($\geq0$)
we define two slash operators on a function
$\phi:\mathfrak{H}\times\mathbb{C}\rightarrow\mathbb{C}$ as
\begin{eqnarray}
\bigg(\phi|_{k,m}\left[\begin{matrix}a&b\\c&d\end{matrix}\right]\bigg)(\tau,z)&:=&(c\tau+d)^{-k}
e^m\bigg(\frac{-cz^2}{c\tau+d}\bigg)\phi\bigg(\frac{a\tau+b}{c\tau+d},\frac{z}{c\tau+d}\bigg)
\quad\bigg(\begin{pmatrix}a&b\\c&d\end{pmatrix}\in\mathrm{SL}_2(\mathbb{R})\bigg)\nonumber\\
(\phi|_m\left[\begin{matrix}\lambda &
\mu\end{matrix}\right])(\tau,z)&:=&e^m(\lambda^2\tau+2\lambda
z+\lambda\mu)\phi(\tau,z+\lambda\tau+\mu)\quad((\begin{matrix}\lambda&\mu\end{matrix})\in\mathbb{R}^2)\label{secondslash}
\end{eqnarray}
where $e^m(x):=\exp(2\pi imx)$ (and $e(x):=e^1(x)$). Then we have
the relations
\begin{eqnarray}
(\phi|_{k,m}M)|_{k,m}M'&=&\phi|_{k,m}(MM')\quad(M,~M'\in\mathrm{SL}_2(\mathbb{R}))\nonumber\\
(\phi|_m X)|_m
X'&=&e^m\bigg(\det\begin{pmatrix}X\\X'\end{pmatrix}\bigg)\phi|_m(X+X')\quad(X,~X'\in\mathbb{R}^2)\nonumber\\
(\phi|_{k,m}M)|_m(XM)&=&(\phi|_mX)|_{k,m}M\quad(M\in\mathrm{SL}_2(\mathbb{R}),~X\in\mathbb{R}^2).\label{slash3}
\end{eqnarray}
(\cite{E-Z} $\S$1). From now on, we let
$\Gamma_1:=\mathrm{SL}_2(\mathbb{Z})$ throughout this paper. A
\textit{Jacobi form} of \textit{weight} $k$ and \textit{index} $m$
on a subgroup $\Gamma\subset\Gamma_1$ of finite index is a
holomorphic function
$\phi:\mathfrak{H}\times\mathbb{C}\rightarrow\mathbb{C}$ satisfying
\begin{itemize}
\item[(i)] $\phi|_{k,m}M=\phi$ for all $M\in\Gamma$,
\item[(ii)] $\phi|_m X=\phi$ for all $X\in\mathbb{Z}^2$,
\item[(iii)] for each $M\in\Gamma_1$, $\phi|_{k,m}M$ has a Fourier
development of the form
\begin{equation*}
\sum_{\begin{smallmatrix}n,r\in\mathbb{Z}\\
r^2\leq 4nm\end{smallmatrix}
}c(n,r)q^n\zeta^r\quad(q:=e(\tau),~\zeta:=e(z)).
\end{equation*}
\end{itemize}
The $\mathbb{C}$-vector space of all such functions $\phi$ is
denoted by $J_{k,m}(\Gamma)$. Then, as is well-known,
$\sum_{k,m}J_{k,m}(\Gamma)$ forms a bigraded ring (\cite{E-Z}
Theorem 1.5). And, Eichler-Zagier further developed the following
theorems.

\begin{theorem}[\cite{E-Z} Theorem 1.3]\label{EZtheorem}
Let $\phi\in J_{k,m}(\Gamma)$ and
$(\begin{matrix}\lambda&\mu\end{matrix})\in\mathbb{Q}^2$. Then the
function
\begin{equation*}
(\phi|_m\left[\begin{matrix}\lambda&\mu\end{matrix}\right])(\tau,0)=e^m(\lambda^2\tau)\phi(\tau,\lambda\tau+\mu)
\end{equation*}
is a modular form of weight $k$ on some subgroup of $\Gamma$ of
finite index depending only on $\Gamma$ and
$(\begin{matrix}\lambda&\mu\end{matrix})$ (in the sense of
\cite{Koblitz}).
\end{theorem}

\begin{theorem}[\cite{E-Z} Theorem 1.1]\label{EZtheorem2}
$J_{k,m}(\Gamma)$ is a finite-dimensional space.
\end{theorem}

On the other hand, a Jacobi form of weight $k$ and index $0$ is
independent of $z$ by the third condition for a Jacobi form, and
hence it is simply an ordinary modular form of weight $k$ in $\tau$.
In this paper we shall first modify the slash operator described in
(\ref{secondslash}) when $m=0$, and define so called
\textit{modified Jacobi forms} of \textit{weight $k$} and
\textit{index $0$}. And we shall obtain an analogue of Theorem
\ref{EZtheorem} (Theorem \ref{main}). This construction of modified
Jacobi forms is in fact motivated by the Weierstrass
$\sigma$-function, which will give us Klein forms as nearly
holomorphic modular forms of integral weight by virtue of Theorem
\ref{main} (Remark \ref{Klein}).
\par
We shall further construct certain subspaces of modified Jacobi
forms of index $0$ and show that they are indeed finite-dimensional
(Theorem \ref{main2}), which could be an analogue of Theorem
\ref{EZtheorem2}.

\section {Modification of a slash operator}

For a lattice
$L=[\omega_1,\omega_2]=\mathbb{Z}\omega_1+\mathbb{Z}\omega_2$ in
$\mathbb{C}$ the \textit{Weierstrass $\wp$-function} is defined by
\begin{equation*}
\wp(z,L):=\frac{1}{z^2}+\sum_{\omega\in
L-\{0\}}\bigg(\frac{1}{(z-\omega)^2}-\frac{1}{\omega^2}\bigg),
\end{equation*}
and the \textit{Weierstrass $\sigma$-function} is defined as
\begin{equation*}
\sigma(z,L):=z\prod_{\omega\in
L-\{0\}}\bigg(1-\frac{z}{\omega}\bigg)e^{(z/\omega)+(z/\omega)^2/2}.
\end{equation*}
Taking the logarithmic derivative we come up with the
\textit{Weierstrass $\zeta$-function}
\begin{equation*}
\zeta(z,L):=\frac{\sigma'(z,L)}{\sigma(z,L)}=\frac{1}{z}+\sum_{\omega\in
L-\{0\}}\bigg(\frac{1}{z-\omega}+\frac{1}{\omega}+\frac{z}{\omega^2}\bigg).
\end{equation*}
Differentiating the function $\zeta(z+\omega,L)-\zeta(z,L)$ for
$\omega\in L$ results in $0$ because
$\frac{d}{dz}\zeta(z,L)=-\wp(z,L)$ and the $\wp$-function is
periodic with respect to $L$. Hence there is a constant
$\eta(\omega,L)$ satisfying
$\zeta(z+\omega,L)=\zeta(z,L)+\eta(\omega,L)$. So we define the
\textit{Weierstrass $\eta$-function} by $\mathbb{R}$-linearity,
namely, if $z=r_1\omega_1+r_2\omega_2$ with $r_1$,
$r_2\in\mathbb{R}$, then
\begin{equation*}
\eta(z,L):=r_1\eta(\omega_1,L)+ r_2\eta(\omega_2,L).
\end{equation*}
We further define a function $\psi(z,L)$ on $\mathbb{C}$ by
\begin{equation*}
\psi(z,L):=\left\{\begin{array}{ll}-1 & \textrm{if}~z\in L-2L\\
1 & \textrm{otherwise}.\end{array}\right.
\end{equation*}
If $X=(\begin{matrix}\lambda &\mu\end{matrix})\in\mathbb{R}^2$ is
fixed, then the value $\psi(\lambda\tau+\mu,[\tau,1])$  does not
depend on $\tau\in\mathfrak{H}$. Therefore we will simply write
$\psi(X)$ for $\psi(\lambda\tau+\mu,[\tau,1])$.

\begin{lemma}\label{transform}
Let $L$ be a lattice in $\mathbb{C}$.
\begin{itemize}
\item[(i)]
$\sigma(z,L)$, $\eta(z,L)$ and $\psi(z,L)$ are homogeneous of degree
$1$, $-1$ and $0$, respectively. Namely, for any
$\lambda\in\mathbb{C}-\{0\}$ we have
\begin{equation*}
\sigma(\lambda z,\lambda L)=\lambda\sigma(z,L),~\eta(\lambda
z,\lambda L)=\frac{1}{\lambda}\eta(z,L)~\textrm{and}~ \psi(\lambda
z,\lambda L)=\psi(z,L).
\end{equation*}
\item[(ii)]
If $L=[\omega_1,\omega_2]$ and $(\begin{matrix}\lambda &
\mu\end{matrix})\in\mathbb{Z}^2$, then
\begin{equation*}
\psi(\lambda\omega_1+\mu\omega_2,L)=(-1)^{\lambda\mu+\lambda+\mu}.
\end{equation*}
\item[(iii)] Let $X\in\mathbb{R}^2$ and
$M\in\mathrm{SL}_2(\mathbb{Z})$. Then $\psi(XM)=\psi(X)$.
\end{itemize}
\end{lemma}
\begin{proof}
One can obtain (i) and (ii) directly from the definitions of
$\sigma(z,L)$, $\eta(z,L)$ and $\psi(z,L)$.
\par
Now, let $X=(\begin{matrix}\lambda&\mu\end{matrix})\in\mathbb{R}^2$
and
$M=\begin{pmatrix}a&b\\c&d\end{pmatrix}\in\mathrm{SL}_2(\mathbb{Z})$.
We derive that
\begin{eqnarray*}
\psi(XM)&=&\psi((\lambda a+\mu c)\tau+(\lambda b+\mu
d),[\tau,1])~\textrm{for any
$\tau\in\mathfrak{H}$}\\
&=&\psi((\lambda a+\mu c)\tau+(\lambda b+\mu d),[a\tau+b,c\tau+d])
~\textrm{by the fact $[a\tau+b,c\tau+d]=[\tau,1]$}\\
&=&\psi(\lambda(a\tau+b)+\mu(c\tau+d),[a\tau+b,c\tau+d])\\
&=&\psi\bigg(\lambda\frac{a\tau+b}{c\tau+d}+\mu,\bigg[\frac{a\tau+b}{c\tau+d},
1\bigg]\bigg)~\textrm{by (i)}\\
&=&\psi(\lambda\tau+\mu,[\tau,1])=\psi(X).
\end{eqnarray*}
This proves (iii).
\end{proof}

\begin{lemma}[Legendre Relation]\label{Legendre}
Let $L=[\omega_1,\omega_2]$ be a lattice in $\mathbb{C}$ with
$\omega_1/\omega_2\in\mathfrak{H}$. Then we have
\begin{equation*}
\eta(\omega_2,L)\omega_1-\eta(\omega_1,L)\omega_2=2\pi i.
\end{equation*}
\end{lemma}
\begin{proof}
See \cite{Lang} p. 241 or \cite{Silverman} p. 41.
\end{proof}

Let $k$ be an integer and
$\phi:\mathfrak{H}\times\mathbb{C}\rightarrow\mathbb{C}$ be a
function. We define two operators $|'_k$ and $|''_k$ on $\phi$ as
follows.
\begin{eqnarray*}
(\phi|'_k M)(\tau,z)&:=&(\phi|_{k,0}M)(\tau,z)\quad(M\in\mathrm{SL}_2(\mathbb{R}))\\
(\phi|''_k
\left[\begin{matrix}\lambda&\mu\end{matrix}\right])(\tau,z)&:=&
\bigg(\psi(\begin{matrix}\lambda&\mu\end{matrix})\exp(\eta(\lambda\tau+\mu,[\tau,1])(z+\tfrac{1}{2}(\lambda\tau+\mu)))\bigg)^k\\
&&\times(\phi|_0
\left[\begin{matrix}\lambda&\mu\end{matrix}\right])(\tau,z)
\quad((\begin{matrix}\lambda&\mu\end{matrix})\in\mathbb{R}^2).
\end{eqnarray*}

\begin{proposition}\label{split}
\begin{itemize}
\item[(i)]
For any $X$, $X'\in\mathbb{R}^2$ we get
\begin{equation*}
(\phi|''_k X)|''_k X'=\bigg(\psi(X)\psi(X') \psi(X+X')
e\bigg(\frac{1}{2}\det\begin{pmatrix}X'\\X\end{pmatrix}\bigg)\bigg)^k(\phi|''_k
(X+X')).
\end{equation*}
In particular, if $X$, $X'\in\mathbb{Z}^2$, then
\begin{equation*}
(\phi|''_k X)|''_k X'=\phi|''_k (X+X').
\end{equation*}
\item[(ii)] For any $M\in\mathrm{SL}_2(\mathbb{Z})$ and $X\in\mathbb{R}^2$ we have
\begin{equation*}
(\phi|'_k M)|''_k (XM)= (\phi|''_k X)|'_k M.
\end{equation*}
\end{itemize}
\end{proposition}
\begin{proof}
(i) Let $X=(\begin{matrix}\lambda & \mu\end{matrix})$ and
$X'=(\begin{matrix}\lambda' & \mu'\end{matrix})$. If we set
$\omega=\lambda\tau+\mu$ and $\omega'=\lambda'\tau+\mu'$, then
\begin{eqnarray*}
&&(\phi|''_kX)|''_k X'\\
&=&\bigg(\bigg(\psi(X)\exp(\eta(\omega,[\tau,1])(z+\tfrac{1}{2}\omega))\bigg)^k\phi(\tau,z+\omega)\bigg)|''_k
X'\\
&=&\bigg(\psi(X)\psi(X')
\exp(\eta(\omega,[\tau,1])(z+\omega'+\tfrac{1}{2}\omega)
+\eta(\omega',[\tau,1])(z+\tfrac{1}{2}\omega'))\bigg)^k
\phi(\tau,z+\omega+\omega')\\
&=&\bigg(\psi(X)\psi(X') \psi(X+X')
\exp(\tfrac{1}{2}(-\eta(\omega',[\tau,1])\omega+\eta(\omega,[\tau,1])\omega'))\bigg)^k
(\phi|''_k(X+X'))\\
&=&\bigg(\psi(X)\psi(X)\psi(X+X')
\exp(\tfrac{1}{2}(\lambda'\mu-\lambda\mu')(\tau\eta(1,[\tau,1])-\eta(\tau,[\tau,1]))\bigg)^k
(\phi|''_k (X+X'))\\
&=&\bigg(\psi(X)\psi(X')\psi(X+X')\exp(\pi
i(\lambda'\mu-\lambda\mu'))\bigg)^k(\phi|''_k (X+X'))~\textrm{by
Lemma \ref{Legendre}}.
\end{eqnarray*}
This yields the first part of (i). Moreover, if $X$,
$X'\in\mathbb{Z}^2$, then
\begin{eqnarray*}
&&\psi(X)\psi(X')\psi(X+X')
\exp(\pi i(\lambda'\mu-\lambda\mu'))\\
&=&
(-1)^{\lambda\mu+\lambda+\mu}(-1)^{\lambda'\mu'+\lambda'+\mu'}(-1)^{(\lambda+\lambda')(\mu+\mu')+(\lambda+\lambda')+(\mu+\mu')}
(-1)^{\lambda'\mu-\lambda\mu'}=1~\textrm{by Lemma
\ref{transform}(ii)},
\end{eqnarray*}
which proves the second part of (i).\\
(ii) Let $M=\begin{pmatrix}a&b\\c&d\end{pmatrix}$,
$X=(\begin{matrix}\lambda &\mu\end{matrix})$ and
$(\begin{matrix}\lambda' &\mu'\end{matrix})=XM$. Then
\begin{eqnarray*}
&&(\phi|'_k M)|''_k(XM)\\
&=&\bigg(\psi(XM)
\exp(\eta(\lambda'\tau+\mu',[\tau,1])(z+\tfrac{1}{2}(\lambda'\tau+\mu'))\bigg)^k
((\phi|_{k,0}M)|_0(XM))\\
&=&\bigg(\psi(X)\exp(\eta(\lambda'\tau+\mu',[\tau,1])
(z+\tfrac{1}{2}(\lambda'\tau+\mu'))\bigg)^k((\phi|_0 X)|_{k,0}
M)~\textrm{by Lemma \ref{transform}(iii) and (\ref{slash3})}.
\end{eqnarray*}
On the other hand, we achieve that
\begin{eqnarray*}
&&(\phi|''_k X)|'_k M\\
&=&
\bigg(\bigg(\psi(X)\exp(\eta(\lambda\tau+\mu,[\tau,1])(z+\tfrac{1}{2}(\lambda\tau+\mu)))\bigg)^k
\phi|_0 X\bigg)|_{k,0}M\\
&=&\bigg(\psi(X)
\exp\bigg(\eta\bigg(\lambda\frac{a\tau+b}{c\tau+d}+\mu,\bigg[\frac{a\tau+b}{c\tau+d},1\bigg]\bigg)
\bigg(\frac{z}{c\tau+d}+\frac{1}{2}\bigg(\lambda\frac{a\tau+b}{c\tau+d}+\mu\bigg)\bigg)\bigg)\bigg)^k
((\phi|_0X)|_{k,0}M)
\\
&=&\bigg(\psi(X)
\exp(\eta(\lambda(a\tau+b)+\mu(c\tau+d),[a\tau+b,c\tau+d])
(z+\tfrac{1}{2}(\lambda(a\tau+b)+\mu(c\tau+d)))\bigg)^k\\
&&\times(\phi|_0X)|_{k,0}M~\textrm{by Lemma \ref{transform}(i)}\\
&=&\bigg(\psi(X)\exp(\eta(\lambda'\tau+\mu',[\tau,1])
(z+\tfrac{1}{2}(\lambda'\tau+\mu'))\bigg)^k
((\phi|_0X)|_{k,0}M)~\textrm{by the
fact}~[a\tau+b,c\tau+d]=[\tau,1].
\end{eqnarray*}
Hence we obtain the assertion (ii).
\end{proof}

Define a function $\rho(\tau,z)$ by
\begin{equation*}
\rho(\tau,z):=\exp(\tfrac{1}{2}\eta(1,[\tau,1])z^2-\pi iz).
\end{equation*}
Since $\eta(1,[\tau,1])$ has the expansion
\begin{equation*}
\eta(1,[\tau,1])=\frac{(2\pi
i)^2}{12}\bigg(-1+24\sum_{n=1}^\infty\frac{nq^n}{1-q^n}\bigg)
\end{equation*}
(\cite{Lang} p. 249), $\rho(\tau,z)$ is a holomorphic function on
$\mathfrak{H}\times\mathbb{C}$.

\begin{definition}
A \textit{modified Jacobi form} of \textit{weight} $k$ and
\textit{index $0$} on a subgroup $\Gamma\subset\Gamma_1$ of finite
index is a holomorphic function
$\phi:\mathfrak{H}\times\mathbb{C}\rightarrow\mathbb{C}$ satisfying
\begin{itemize}
\item[(i)] $\phi|'_k M=\phi$ for all $M\in\Gamma$,
\item[(ii)] $\phi|''_k X=\phi$ for all $X\in\mathbb{Z}^2$,
\item[(iii)] for each $M\in\Gamma_1$, the function  $\rho^k(\phi|'_k M)$ has a Fourier
development of the form
\begin{equation*}
\sum_{n\geq0}\sum_{|r|\leq r_0(n)}c(n,r)q^n\zeta^r
\end{equation*}
for some increasing sequence $\{r_0(n)\}_{n\geq0}$ of nonnegative
integers such that
\begin{equation*}
\frac{r_0(n)}{n}\rightarrow0~\textrm{as}~n\rightarrow\infty.
\end{equation*}
Here we allow that $c(n,r)$ could be zero even if $|r|\leq r_0(n)$.
\end{itemize}
In what follows, we denote the $\mathbb{C}$-vector space of all such
functions $\phi$ by $J_k(\Gamma)$. And we let
$J_*(\Gamma):=\sum_{k\in\mathbb{Z}}J_k(\Gamma)$.
\end{definition}

\begin{proposition}\label{graded}
If $\Gamma$ is a subgroup of $\Gamma_1$ of finite index, then
$J_*(\Gamma)$ is a graded ring.
\end{proposition}
\begin{proof}
Let $\phi\in J_{k}(\Gamma)$ and $\phi'\in J_{k'}(\Gamma)$ for some
$k$, $k'\in\mathbb{Z}$. If $M\in\Gamma$ and $X\in\mathbb{Z}^2$, then
one can readily show by the definitions of slash operators that
$(\phi\phi')|'_{k+k'}M=\phi\phi'$ and
$(\phi\phi')|''_{k+k'}X=\phi\phi'$.
\par
Now let $M\in\Gamma_1$ and suppose that $\rho^k(\phi|'_k M)$ and
$\rho^{k'}(\phi'|'_{k'}M)$ have Fourier developments
\begin{equation*}
\sum_{n\geq0}\sum_{|r|\leq r_0(n)}c(n,r)q^n\zeta^r
~\textrm{and}~\sum_{m\geq0}\sum_{|s|\leq s_0(m)}d(m,s)q^m\zeta^s,
\end{equation*}
respectively. Here $\{r_0(n)\}_{n\geq0}$ and $\{s_0(m)\}_{m\geq0}$
are increasing sequences of nonnegative integers such that both
$r_0(n)/n$ and $s_0(m)/m$ tend to $0$ as $n$, $m\rightarrow\infty$.
We then have
\begin{eqnarray*}
\rho^{k+k'}((\phi\phi')|'_{k+k'}M)&=&\rho^k(\phi|'_k
M)\rho^{k'}(\phi'|'_{k'}M)~\textrm{by the
definition of $|'$}\\
&=&\sum_{\begin{smallmatrix}\ell=n+m\\
n\geq0,~m\geq0\end{smallmatrix}}\sum_{\begin{smallmatrix}t=r+s
\\|r|\leq r_0(n),~|s|\leq s_0(m)\end{smallmatrix}}
c(n,r)d(m,s)q^{\ell}\zeta^t.
\end{eqnarray*}
Note that for a given $\ell\geq0$ there are only finitely many
nonnegative integers $n$ and $m$ such that $\ell=n+m$. So there are
only finitely many integers $t$ which contribute terms of the form
$q^\ell\zeta^t$. Furthermore, we have
\begin{equation*}
\frac{|t|}{\ell}\leq\frac{r_0(n)+s_0(m)}{n+m}
\leq2\max\bigg(\frac{r_0(n+m)}{n+m},\frac{s_0(n+m)}{n+m}\bigg)\rightarrow0~\textrm{as}~
n+m\rightarrow\infty.
\end{equation*}
Hence $\phi\phi'$ satisfies the third condition for a modified
Jacobi form of weight $k+k'$. This proves the proposition.
\end{proof}

\begin{proposition}\label{zero}
Let $\phi$ be a modified Jacobi form of weight $k$. Assume that  the
function $z\mapsto\phi(\tau_0,z)$ is not identically zero for a
fixed point $\tau_0\in\mathfrak{H}$. Let $F$ be a fundamental domain
for the torus $\mathbb{C}/[\tau_0,1]$, whose boundary does not have
any zeros of $\phi(\tau_0,z)$(, such $F$ always exists). Then
$\phi(\tau_0,z)$ has exactly $-k$ zeros (counting multiplicity) in
$F$.
\end{proposition}
\begin{proof}
It follows from the second condition for a modified Jacobi form
$\phi$ that
\begin{equation}\label{2ndcondition}
\bigg(\psi(\begin{matrix}\lambda&\mu\end{matrix})\exp(\eta(\lambda\tau+\mu,[\tau,1])(z+\tfrac{1}{2}(\lambda\tau+\mu)))\bigg)^k
\phi(\tau,z+\lambda\tau+\mu)=\phi(\tau,z)
\quad((\begin{matrix}\lambda&\mu\end{matrix})\in\mathbb{Z}^2).
\end{equation}
Differentiating the above equation with respect to $z$ we have
\begin{equation}\label{differentiation}
\bigg(\psi(\begin{matrix}\lambda&\mu\end{matrix})\exp(\eta(\lambda\tau+\mu,[\tau,1])(z+\tfrac{1}{2}(\lambda\tau+\mu)))\bigg)^k
\bigg(k\eta(\lambda\tau+\mu,[\tau,1])\phi(\tau,z+\lambda\tau+\mu)
+\phi_z(\tau,z+\lambda\tau+\mu)\bigg)=\phi_z(\tau,z)
\end{equation}
where $\phi_z=\frac{d}{dz}\phi$. Dividing the equation
(\ref{differentiation}) by (\ref{2ndcondition}) we obtain
\begin{equation*}
k\eta(\lambda\tau+\mu,[\tau,1])+\frac{\phi_z}{\phi}(\tau,z+\lambda\tau+\mu)=\frac{\phi_z}{\phi}(\tau,z).
\end{equation*}
If we put $(\begin{matrix}\lambda &\mu\end{matrix})=
(\begin{matrix}1&0\end{matrix})$ and
$(\begin{matrix}0&1\end{matrix})$ in the previous equation, then we
get that
\begin{equation}\label{putting}
k\eta(\tau,[\tau,1])+\frac{\phi_z}{\phi}(\tau,z+\tau)=\frac{\phi_z}{\phi}(\tau,z)~\textrm{and}~
k\eta(1,[\tau,1])+\frac{\phi_z}{\phi}(\tau,z+1)=\frac{\phi_z}{\phi}(\tau,z),
\end{equation}
respectively. Now, we set $\partial F=\partial F_1+\partial
F_2+\partial F_3+\partial F_4$ as follows:
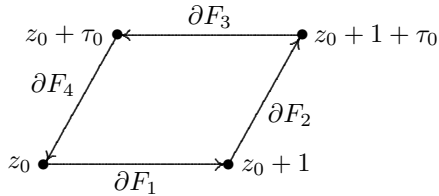
\begin{equation*}
\begindc{\undigraph}[7]
\obj(1,1){$z_0$}[\west]

\obj(11,1){$z_0+1$}[\east]

\obj(15,8){$z_0+1+\tau_0$}[\east]

\obj(5,8){$z_0+\tau_0$}[\west]

\mor{$z_0$}{$z_0+1$}{$\partial F_1$}[\atright,\solidarrow]

\mor{$z_0+1$}{$z_0+1+\tau_0$}{$\partial F_2$}[\atright,\solidarrow]

\mor{$z_0+1+\tau_0$}{$z_0+\tau_0$}{$\partial
F_3$}[\atright,\solidarrow]

\mor{$z_0+\tau_0$}{$z_0$}{$\partial F_4$}[\atright,\solidarrow]
\enddc
\end{equation*}
By the Residue Theorem we derive that the number of zeros of
$\phi(\tau_0,z)$ in $F$ is equal to
\begin{eqnarray*}
\frac{1}{2\pi i} \oint_{\partial
F}\frac{\phi_z}{\phi}(\tau_0,z)dz&=&\bigg( \frac{1}{2\pi i}
\oint_{\partial F_1}\frac{\phi_z}{\phi}(\tau_0,z)dz +\frac{1}{2\pi
i} \oint_{\partial F_3}\frac{\phi_z}{\phi}(\tau_0,z)dz \bigg)\\
&&+ \bigg( \frac{1}{2\pi i} \oint_{\partial
F_2}\frac{\phi_z}{\phi}(\tau_0,z)dz
+\frac{1}{2\pi i} \oint_{\partial F_4}\frac{\phi_z}{\phi}(\tau_0,z)dz \bigg)\\
&=&\bigg( \frac{1}{2\pi i} \oint_{\partial
F_1}\frac{\phi_z}{\phi}(\tau_0,z)dz -\frac{1}{2\pi i}
\oint_{\partial F_1}\frac{\phi_z}{\phi}(\tau_0,z+\tau_0)dz \bigg)
\\
&&+ \bigg(-\frac{1}{2\pi i} \oint_{\partial
F_4}\frac{\phi_z}{\phi}(\tau_0,z+1)dz
+\frac{1}{2\pi i} \oint_{\partial F_4}\frac{\phi_z}{\phi}(\tau_0,z)dz \bigg)\\
&=&\frac{1}{2\pi i}\oint_{\partial F_1}k\eta(\tau_0,[\tau_0,1])dz +
\frac{1}{2\pi i}\oint_{\partial F_4}k\eta(1,[\tau_0,1])dz~\textrm{by
(\ref{putting})}\\
&=&\frac{k}{2\pi
i}(\eta(\tau_0,[\tau_0,1])-\tau_0\eta(1,[\tau_0,1]))=-k~\textrm{by
Lemma \ref{Legendre}}.
\end{eqnarray*}
This completes the proof.
\end{proof}

\begin{corollary}
Let $\Gamma$ be a subgroup of $\Gamma_1$ of finite index. For all
positive integers $k$ we have $J_k(\Gamma)=\{0\}$.
\end{corollary}
\begin{proof}
Assume that there exists a nonzero element $\phi$ of $J_k(\Gamma)$
for some $k>0$. Take a point $\tau_0\in\mathfrak{H}$ such that
$\phi(\tau_0,z)$ is not identically zero. And, consider a
fundamental domain $F$ for $\mathbb{C}/[\tau_0,1]$ whose boundary
has no zeros of $\phi(\tau_0,z)$. Then $\phi(\tau_0,z)$ has $-k<0$
zeros in $F$ by Proposition \ref{zero}, which is impossible. Hence
$J_k(\Gamma)=\{0\}$ for all $k>0$, as desired.
\end{proof}

\section{Construction of nearly holomorphic modular forms}

In this section we shall show that through a modified Jacobi form
one can generate nearly holomorphic modular forms of integral
weight.

\begin{lemma}\label{Legendre2}
If $(\begin{matrix}\lambda&\mu\end{matrix})\in\mathbb{R}^2$, then
\begin{equation*}
\exp(\eta(\lambda\tau+\mu,[\tau,1])(z+\tfrac{1}{2}(\lambda\tau
+\mu)))\rho(\tau,z)\rho(\tau,z+\lambda\tau+\mu)^{-1}=
e(\tfrac{1}{2}\mu(1-\lambda))q^{\lambda(1-\lambda)/2}\zeta^{-\lambda}.
\end{equation*}
\end{lemma}
\begin{proof}
We achieve that
\begin{eqnarray*}
&&\exp(\eta(\lambda\tau+\mu,[\tau,1])(z+\tfrac{1}{2}(\lambda\tau
+\mu)))\rho(\tau,z)\rho(\tau,z+\lambda\tau+\mu)^{-1}\\
&=&\exp\bigg((\lambda\eta(\tau,[\tau,1])+\mu\eta(1,[\tau,1]))
(z+\tfrac{1}{2}(\lambda\tau+\mu))
+\tfrac{1}{2}\eta(1,[\tau,1])z^2-\pi
iz\\
&&-\tfrac{1}{2}\eta(1,[\tau,1])
(z+\lambda\tau+\mu)^2+\pi i(z+\lambda\tau+\mu)\bigg)\\
&=&\exp\bigg((\tau\eta(1,[\tau,1])-\eta(\tau,[\tau,1]))(-\lambda
-\tfrac{1}{2}\lambda^2\tau-\tfrac{1}{2}\lambda\mu)+\pi i\lambda\tau
+\pi i\mu\bigg)\\
&=&\exp\bigg(2\pi i(-\lambda
-\tfrac{1}{2}\lambda^2\tau-\tfrac{1}{2}\lambda\mu)+\pi i\lambda\tau
+\pi i\mu\bigg)~\textrm{by Lemma \ref{Legendre}}\\
&=&e(\tfrac{1}{2}\mu(1-\lambda))q^{\lambda(1-\lambda)/2}\zeta^{-\lambda}
\end{eqnarray*}
as desired.
\end{proof}

\begin{theorem}\label{main}
Let $\Gamma$ be a subgroup of $\Gamma_1$ of finite index. Let
$\phi\in J_k(\Gamma)$ for some integer $k$. Then for
$X\in\mathbb{Q}^2$ the function
\begin{equation*}
\phi_X(\tau):=(\phi|''_k X)(\tau,0)
\end{equation*}
is a nearly holomorphic modular form (that is, holomorphic on
$\mathfrak{H}$ and meromorphic at every cusp) of weight $k$ on some
subgroup of $\Gamma$ of finite index depending only on $\Gamma$ and
$X$.
\end{theorem}
\begin{proof}
Let $X=(\begin{matrix}\lambda &\mu\end{matrix})$. For any $X'\in
2\mathbb{Z}^2$ we have
\begin{eqnarray*}
\phi_{X+X'}(\tau)&=&(\phi|''_k(X+X'))(\tau,0)
=(\phi|''_k(X'+X)(\tau,0)\\
&=&\bigg(\psi(X')\psi(X) \psi(X'+X)
e\bigg(\frac{1}{2}\det\begin{pmatrix}X\\X'\end{pmatrix}\bigg)\bigg)^{-k}
((\phi|''_k X')|''_k X)(\tau,0)~\textrm{by Proposition \ref{split}(i)}\\
&=&\bigg(\psi(X) \psi(X)
e\bigg(\frac{1}{2}\det\begin{pmatrix}X\\X'\end{pmatrix}\bigg)\bigg)^{-k}
((\phi|''_k X')|''_k X)(\tau,0)~\textrm{because
$X'\in 2\mathbb{Z}^2$}\\
&=&e\bigg(-\frac{k}{2}\det\begin{pmatrix}X\\X'\end{pmatrix}\bigg)
\phi_X(\tau)~\textrm{from the second condition for $\phi(\tau,z)$}.
\end{eqnarray*}
On the other hand, for any
$M=\begin{pmatrix}a&b\\c&d\end{pmatrix}\in\Gamma$ we deduce that
\begin{eqnarray*}
(c\tau+d)^{-k}\phi_X\bigg(\frac{a\tau+b}{c\tau+d}\bigg)&=&
((\phi|''_k X)|'_k M)(\tau,0)~\textrm{by the definitions of $\phi_X$ and slash operators}\\
&=&((\phi|'_k M)|''_k (XM))(\tau,0)~\textrm{by Proposition \ref{split}(ii)}\\
&=&(\phi|''_k (XM))(\tau,0)~\textrm{by the first condition for
$\phi(\tau,z)$}\\
&=&\phi_{XM}(\tau).
\end{eqnarray*}
Thus the above observations indicate that $\phi_X(\tau)$ behaves
like a modular form with respect to the congruence subgroup
\begin{eqnarray*}
&&\bigg\{M\in\Gamma~:~XM\equiv X\pmod{2\mathbb{Z}^2},~\frac{k}{2}
\cdot\det\begin{pmatrix}X\\XM-X\end{pmatrix}\in\mathbb{Z}
\bigg\}\\
&=&\bigg\{\begin{pmatrix} a&b\\c&d
\end{pmatrix}\in\Gamma~:~
(a-1)\lambda+c\mu,~b\lambda+(d-1)\mu,~k(b\lambda^2+(d-a)\lambda\mu-c\mu^2)\in
2\mathbb{Z} \bigg\}
\end{eqnarray*}
which contains $\Gamma\cap\Gamma(\frac{2N^2}{\gcd(N,k)})$ if $X\in
N^{-1}\mathbb{Z}^2$ for some integer $N\geq1$ and
$\Gamma(\frac{2N^2}{\gcd(N,k)})$ is the principal congruence
subgroup of level $\frac{2N^2}{\gcd(N,k)}$.
\par
Next, let $M=\begin{pmatrix}a&b\\c&d\end{pmatrix}\in\Gamma_1$ and
suppose that $\rho^k(\phi|'_k M)$ has a Fourier development
\begin{equation*}
\sum_{n\geq0}\sum_{|r|\leq r_0(n)}c(n,r)q^n\zeta^r
\end{equation*}
such that $r_0(n)/n\rightarrow0$ as $n\rightarrow\infty$. Then we
get that
\begin{eqnarray}
&&(c\tau+d)^{-k}\phi_X\bigg(\frac{a\tau+b}{c\tau+d}\bigg)\nonumber\\
&=&((\phi|''_k X)|'_k M)(\tau,0)~\textrm{by the definitions of $\phi_X$ and slash operators}\nonumber\\
&=&((\phi|'_k M)|''_k (XM))(\tau,0)~\textrm{by Proposition \ref{split}(ii)}\nonumber\\
&=&
\bigg(\psi(XM)\exp(\tfrac{1}{2}\eta(\lambda'\tau+\mu',[\tau,1])(\lambda'\tau+\mu'))\bigg)^k
(\phi|'_k M)(\tau,\lambda'\tau+\mu')~\textrm{where}~
(\begin{matrix}\lambda'&\mu'\end{matrix})=XM\nonumber\\
&=&\bigg(\psi(XM)\exp(\tfrac{1}{2}\eta(\lambda'\tau+\mu',[\tau,1])(\lambda'\tau+\mu'))\bigg)^k
\bigg(\rho^{-k}\sum_{n\geq0}\sum_{|r|\leq r_0(n)}c(n,r)q^n\zeta^r\bigg)(\tau,\lambda'\tau+\mu')\nonumber\\
&=&\bigg(\psi(XM)e(\tfrac{1}{2}\mu'(1-\lambda'))q^{\lambda'(1-\lambda')/2}\bigg)^k\sum_{n\geq0}\sum_{|r|\leq
r_0(n)} c(n,r)e(\mu'r)q^{n+\lambda'r}~\textrm{by Lemma
\ref{Legendre2}}.\label{sum}
\end{eqnarray}
Now, we observe that
\begin{equation*}
n+\lambda' r\geq
n-|\lambda'|r_0(n)=n\bigg(1-|\lambda'|\frac{r_0(n)}{n}\bigg)\rightarrow\infty
~\textrm{as}~n\rightarrow\infty,
\end{equation*}
from which it follows that for a given rational number $\ell$ there
are only finitely many $n$ and $r$ which contributes the term
$q^{\ell}$ in (\ref{sum}). Hence $\phi_X(\tau)$ is meromorphic at
each cusp. This completes the proof.
\end{proof}

We are ready to introduce well-known Klein forms in view of
Weierstrass $\sigma$-function which will be a concrete example of
modified Jacobi form.

\begin{lemma}\label{sigmafunction}
\begin{itemize}
\item[(i)] Let $L$ be a lattice in $\mathbb{C}$.
If $\omega\in L$, then
\begin{equation*}
\frac{\sigma(z+\omega,L)}{\sigma(z,L)}=\psi(\omega,L)\exp(\eta(\omega,L)(z+\tfrac{1}{2}\omega)).
\end{equation*}
\item[(ii)]
The function $\sigma(\tau,z):=\sigma(z,[\tau,1])$ has the infinite
product expansion
\begin{equation*}
\sigma(\tau,z)=-\frac{1}{2\pi
i}\rho(\tau,z)(1-\zeta)\prod_{n=1}^\infty
\frac{(1-q^n\zeta)(1-q^{n}\zeta^{-1})}{(1-q^n)^2}.
\end{equation*}
\end{itemize}
\end{lemma}
\begin{proof}
See \cite{Lang} p. 241, p. 247 or \cite{Silverman} p. 44, p. 53.
\end{proof}

\begin{theorem}
The function $\sigma(\tau,z)$ belongs to $J_{-1}(\Gamma_1)$.
\end{theorem}
\begin{proof}
First, we note that $\sigma(\tau,z)$ is a holomorphic function by
Lemma \ref{sigmafunction}(ii).
\par
Let $M=\begin{pmatrix}a&b\\c&d\end{pmatrix}\in\Gamma_1$. Then we
derive
\begin{eqnarray*}
(\sigma|'_{-1}M)(\tau,z)&=&(c\tau+d)\sigma\bigg(\frac{z}{c\tau+d},\bigg[\frac{a\tau+b}{c\tau+d},1\bigg]\bigg)\\
&=&\sigma(z,[a\tau+b,c\tau+d])~\textrm{by Lemma
\ref{transform}(i)}\\
&=&\sigma(\tau,z)~\textrm{by the fact $[a\tau+b,c\tau+d]=[\tau,1]$}.
\end{eqnarray*}
\par
And, for $X=(\begin{matrix}\lambda &\mu\end{matrix})\in\mathbb{Z}^2$
we achieve that
\begin{eqnarray*}
(\sigma|''_{-1}X)(\tau,z)&=&\bigg(\psi(\begin{matrix}\lambda&\mu\end{matrix})
\exp(\eta(\lambda\tau+\mu,[\tau,1])(z+\tfrac{1}{2}(\lambda\tau+\mu)))\bigg)^{-1}
\sigma(\tau,z+\lambda\tau+\mu)\\
&=&\sigma(\tau,z)~\textrm{by Lemma \ref{sigmafunction}(i)}.
\end{eqnarray*}
\par
Finally, let $M\in\Gamma_1$. Then we have
\begin{eqnarray*}
\rho(\tau,z)^{-1}(\sigma|'_{-1}M)(\tau,z)&=&\rho(\tau,z)^{-1}\sigma(\tau,z)~\textrm{by
the first part of the proof}\\
&=&-\frac{1}{2\pi i}(1-\zeta)\prod_{n=1}^\infty
\frac{(1-q^n\zeta)(1-q^{n}\zeta^{-1})}{(1-q^n)^2}~\textrm{by Lemma
\ref{sigmafunction}(ii)}.
\end{eqnarray*}
Hence it is easy to check that $\sigma(\tau,z)$ satisfies the third
condition for a modified Jacobi form by adopting the argument in the
proof of Proposition \ref{graded}. Therefore $\sigma(\tau,z)$
belongs to $J_{-1}(\Gamma_1)$.
\end{proof}

\begin{remark}\label{Klein}
It follows that if $(\begin{matrix}\lambda &\mu\end{matrix})\in
N^{-1}\mathbb{Z}^2$ for some integer $N\geq1$, then the function
\begin{equation*}
\mathfrak{k}_{(\begin{matrix}\lambda &\mu\end{matrix})}(\tau):=
(\sigma|''_{-1}\left[\begin{matrix}\lambda
&\mu\end{matrix}\right])(\tau,0)=
\psi(\begin{matrix}\lambda&\mu\end{matrix})\exp(-\tfrac{1}{2}\eta(\lambda\tau+\mu,[\tau,1])(\lambda\tau+\mu))
\sigma(\tau,\lambda\tau+\mu)
\end{equation*}
is a nearly holomorphic modular form of weight $-1$ on
$\Gamma(2N^2)$ by Theorem \ref{main}. This function is called a
\textit{Klein form} (indexed by $(\begin{matrix}\lambda
&\mu\end{matrix})$) whose infinite product expansion is
\begin{equation*}
\mathfrak{k}_{(\begin{matrix}\lambda
&\mu\end{matrix})}(\tau)=e(\mu(\lambda-1)/2)q^{\lambda(\lambda-1)/2}
(1-e(\mu)q^\lambda) \prod_{n=1}^\infty\frac{(1-e(\mu)q^{n+\lambda})
(1-e(-\mu)q^{n-\lambda})}{(1-q^n)^2}
\end{equation*}
if $(\begin{matrix}\lambda&\mu\end{matrix})\not\in\mathbb{Z}^2$ (and
identically zero otherwise). The modularity of a product of finitely
many Klein forms on $\Gamma(N)$ are intensively studied by Kubert
and Lang (\cite{K-L}). On the other hand, in a recent paper
\cite{E-K-S} the authors present a sufficient condition for a
product of Klein forms to be a nearly holomorphic modular form on
$\Gamma_1(N)$.
\end{remark}

\section {Finite-dimensional subspaces}

In this section we shall consider a family of finite-dimensional
subspaces of $J_k(\Gamma_1)$.

\begin{definition}
Let $k$ be an integer. For each integer $m>0$ we let $J_k^m$ be the
subspace of $J_k$ ($:=J_k(\Gamma_1)$) consisting of $\phi$ for which
$\rho^k\phi=\sum_{n\geq0}\sum_{|r|\leq r_0(n)}c(n,r)q^n\zeta^r$ and
\begin{equation}\label{condition}
\min\{n-r_0(n)~:~n\geq0\}+\frac{k}{8}\geq-m.
\end{equation}
\end{definition}

\begin{remark}
\begin{itemize}
\item[(i)]
Note that since
\begin{equation*}
n-r_0(n)=n\bigg(1-\frac{r_0(n)}{n}\bigg)\rightarrow\infty~\textrm{as}~n\rightarrow\infty,
\end{equation*}
the minimum of $n-r_0(n)$ exists. Thus we have a filtration
$J_k^1\subseteq J_k^2\subseteq\cdots$ and $J_k=\cup_{m=1}^\infty
J_k^m$.
\item[(ii)]
Since $\rho^{-1}\sigma$ has the following Fourier development
\begin{eqnarray*}
-\frac{1}{2\pi
i}(1-\zeta)\prod_{n=1}^\infty\frac{(1-q^n\zeta)(1-q^n\zeta^{-1})}{(1-q^n)^2}
&=&-\frac{1}{2\pi i}\bigg\{
(1-\zeta)+(-\zeta^{-1}+3-3\zeta+\zeta^2)q\\
&&\hspace{0.8cm}+(-3\zeta^{-1}+9-9\zeta+3\zeta^2)q^2\\
&&\hspace{0.8cm}+(\zeta^{-2}-9\zeta^{-1}+22-22\zeta+9\zeta^2-\zeta^3)q^3\\
&&\hspace{0.8cm}+(51-51\zeta-22\zeta^{-1}+22\zeta^2+3\zeta^2-3\zeta^3)q^4\\
&&\hspace{0.8cm}+(9\zeta^{-2}-51\zeta^{-1}+108-108\zeta+51\zeta^2-9\zeta^3)q^5+\cdots
\bigg\},
\end{eqnarray*}
one can readily check that the Weierstrass $\sigma$-function in fact
lies in $J_{-1}^2$.
\end{itemize}
\end{remark}

\begin{proposition}
Let $\phi\in J_k$ for some integer $k$ with
$\rho^k\phi=\sum_{n,r}c(n,r)q^n\zeta^r$. Then we have
\begin{equation}\label{coeffs}
c(n,r)=(-1)^{\lambda
k}c(n-r\lambda-\tfrac{1}{2}k(\lambda^2+\lambda), r+\lambda k)
\end{equation}
for all integers $n$, $r$ and $\lambda$.
\end{proposition}
\begin{proof}
If $(\begin{matrix}\lambda &\mu\end{matrix})\in\mathbb{Z}^2$, then
we derive that
\begin{eqnarray*}
&&\sum_{n,r}c(n,r)q^n\zeta^r=\rho^k\phi(\tau,z)
=\rho^k(\phi|''_k\left[\begin{matrix}\lambda&\mu\end{matrix}\right])\\
&=&\rho(\tau,z)^k
\bigg(\psi(\begin{matrix}\lambda&\mu\end{matrix})\exp(\eta(\lambda\tau+\mu,[\tau,1])(z+\tfrac{1}{2}(\lambda\tau+\mu)))\bigg)^k
\phi(\tau,z+\lambda\tau+\mu)\\
&=&\rho(\tau,z)^k
\bigg(\psi(\begin{matrix}\lambda&\mu\end{matrix})\exp(\eta(\lambda\tau+\mu,[\tau,1])(z+\tfrac{1}{2}(\lambda\tau+\mu)))\bigg)^k
\rho(\tau,z+\lambda\tau+\mu)^{-k}(\rho^k\phi)(\tau,z+\lambda\tau+\mu)\\
&=&\bigg(\psi(\begin{matrix}\lambda&\mu\end{matrix})e(\tfrac{1}{2}\mu(1-\lambda))
q^{\lambda(1-\lambda)/2}\zeta^{-\lambda}\bigg)^k\sum_{n,r}c(n,r)q^{n+r\lambda}\zeta^r
~\textrm{by Lemma \ref{Legendre2}}\\
&=&(-1)^{\lambda
k}\sum_{n,r}c(n,r)q^{n+r\lambda+k\lambda(1-\lambda)/2}\zeta^{r-\lambda
k}~\textrm{by Lemma \ref{transform}(ii)}\\
&=&(-1)^{\lambda
k}\sum_{n',r'}c(n'-r'\lambda-\tfrac{1}{2}k(\lambda^2+\lambda),r'+\lambda
k)q^{n'}\zeta^{r'}\\
&&\textrm{by letting $r':=r-\lambda k$ and
$n':=n+r'\lambda+\tfrac{1}{2}k(\lambda^2+\lambda)$}.
\end{eqnarray*}
We get the assertion by comparing the coefficients of Fourier
developments.
\end{proof}

\begin{remark}\label{relationrk}
If we put $\lambda=1$ in (\ref{coeffs}), then we have
\begin{equation}\label{relation}
c(n,r)=(-1)^kc(n-r-k,r+k).
\end{equation}
\end{remark}

Let $\phi\in J_k$ for some integer $k$ with
$\rho^k\phi=\sum_{n,r}c(n,r)q^n\zeta^r$. Then, for
$X=(\begin{matrix}u&v\end{matrix})\in\mathbb{Q}^2$ we defined in
Theorem \ref{main}
\begin{eqnarray}
&&\phi_X(\tau):=(\phi|''_k X)(\tau,0)\nonumber\\
&=&\bigg\{\bigg(\psi(\begin{matrix}u&v\end{matrix})\exp
(\eta(u\tau+v,[\tau,1])(z+\tfrac{1}{2}(u\tau+v)))\bigg)^k
\rho(\tau,z+u\tau+v)^{-k}(\rho^{k}\phi)(\tau,
z+u\tau+v)\bigg\}(\tau,0)\nonumber\\
&=&\bigg\{
\bigg(\psi(\begin{matrix}u&v\end{matrix})e(\tfrac{1}{2}v(1-u))q^{u(1-u)/2}\zeta^{-u}\bigg)^k
\sum_{n,r}c(n,r)e(vr)q^{n+ur}\zeta^r\bigg\}
(\tau,0)~\textrm{by Lemma \ref{Legendre2}}\nonumber\\
&=&\bigg(\psi(\begin{matrix}u
&v\end{matrix})e(\tfrac{1}{2}v(1-u))\bigg)^k\sum_{n,r} c(n,r)e(v
r)q^{n+ur+ku(1-u)/2}.\label{expression}
\end{eqnarray}

\begin{proposition}\label{01}
Let $X,~Y\in\mathbb{Q}^2$ with $X\equiv Y\pmod{\mathbb{Z}^2}$. Then
$\phi_X(\tau)=\xi\phi_Y(\tau)$ for some root of unity $\xi$.
\end{proposition}
\begin{proof}
If $X=(\begin{matrix}u&v\end{matrix})$ and $Y=(\begin{matrix}
u&v+1\end{matrix})$, then one can readily get from the expression
(\ref{expression}) that $\phi_X(\tau)=\xi\phi_Y(\tau)$ for some root
of unity $\xi$.
\par
Now, let $X=(\begin{matrix}u&v\end{matrix})$ and
$Y=(\begin{matrix}u+1&v\end{matrix})$. We obtain from the expression
(\ref{expression}) that
\begin{eqnarray*}
\phi_Y(\tau)&=& \bigg(\psi(\begin{matrix}u+1
&v\end{matrix})e(-\tfrac{1}{2}vu)\bigg)^k\sum_{n,r} c(n,r)e(v
r)q^{n+(u+1)r-k(u+1)u/2}\\
&=&\bigg(\psi(\begin{matrix}u+1&v\end{matrix})e(-\tfrac{1}{2}vu)\bigg)^k
\sum_{n',r'}c(n'-r'-k,r'+k)e(vr')e(vk)q^{n'+ur'+ku(1-u)/2}\\
&&\textrm{by letting $r':=r-k$ and $n':=n+r'+k$}\\
&=&\bigg(\psi(\begin{matrix}u+1&v\end{matrix})e(-\tfrac{1}{2}vu)\bigg)^k
e(vk)(-1)^k\sum_{n',r'}c(n',r')e(vr')q^{n'+ur'+ku(1-u)/2}
~\textrm{by Remark \ref{relationrk}}\\
&=&\xi\phi_X(\tau)~\textrm{for some root of unity $\xi$}.
\end{eqnarray*}
This proves the proposition.
\end{proof}

Now we are ready to prove our main theorem about dimension.

\begin{theorem}\label{main2}
Let $k<0$ and $m>0$ be integers. Then $J_k^m$ is finite-dimensional.
\end{theorem}
\begin{proof}
Pick any $-k+1$ distinct pairs of
$X_j=(\begin{matrix}u_j&v_j\end{matrix})\in\mathbb{Q}^2$ with
$0<u_j,~v_j<1$. Let $\phi\in J_k^m$ with
$\rho^k\phi=\sum_{n,r}c(n,r)q^n\zeta^r$. Then the functions
\begin{equation*}
\phi_{X_j}(\tau)=(\phi|''_k X_j)(\tau,0)\quad(j=1,\cdots,-k+1)
\end{equation*}
are nearly holomorphic modular forms on some congruence subgroups
$\Gamma_j$ depending on $X_j$ by Theorem \ref{main}. If
$M=\begin{pmatrix} a&b\\c&d\end{pmatrix}\in\Gamma_1$, then we have
\begin{eqnarray}
(c\tau+d)^{-k}\phi_{X_j}\bigg(\frac{a\tau+b}{c\tau+d}\bigg)&=&
(\phi|''_k X_j|'_k M)(\tau,0)~\textrm{by the definitions of slash
operators}\nonumber\\
&=&(\phi|'_k M)|''_k(X_jM)(\tau,0)~\textrm{by Proposition \ref{split}(ii)}\nonumber\\
&=&(\phi|''_k(X_jM))(\tau,0)~\textrm{by the first condition for a modified Jacobi form}\nonumber\\
&=&\phi_{X_jM}(\tau).\label{cusp}
\end{eqnarray}
Set $(\begin{matrix}u_j'&v_j'\end{matrix}):=X_jM$. We may assume
that $0\leq u_j',~v_j'<1$ to estimate
$\mathrm{ord}_q\phi_{X_jM}(\tau)$ by Proposition \ref{01}. Note that
from the expression (\ref{expression}) one can deduce
\begin{equation*}
\phi_{X_jM}(\tau)=\xi \sum_{n\geq0} \sum_{|r|\leq r_0(n)}
c(n,r)e(v_j'r)q^{n+u_j'r+ku_j'(1-u_j')/2}
\end{equation*}
for some root of unity $\xi$. It follows that
\begin{eqnarray}
\mathrm{ord}_q\phi_{X_jM}(\tau)&\geq&
\min\{n+u_j'r+ku_j'(1-u_j')/2~:~n\geq0,~|r|\leq r_0(n)\}\nonumber\\
&\geq&\min\{n-r_0(n)+k/8~:~n\geq0\}~\textrm{because $0\leq
u_j',~v_j'<1$}\nonumber\\
&\geq&-m~\textrm{by the condition (\ref{condition}) for $\phi\in
J_k^m$}.\label{inequality}
\end{eqnarray}
\par
On the other hand, we consider a map
\begin{eqnarray*}
g~:~J_k^m&\rightarrow&\bigoplus_{j=1}^{-k+1}M_{k+12m}(\Gamma_j)\nonumber\\
\phi&\mapsto&(\phi_{X_j}(\tau)\Delta(\tau)^m)_{j=1}^{-k+1}
\end{eqnarray*}
where $M_{k+12m}(\Gamma_j)$ is the space of modular forms of weight
$k+12m$ on $\Gamma_j$, and $\Delta(\tau):=(2\pi
i)^{12}q\prod_{n=1}^\infty(1-q^n)^{24}$ is the modular discriminant
function. As is well-known, each $M_{k+12m}(\Gamma_j)$ is of finite
dimension (\cite{Shimura} Theorem 2.23) and $\Delta(\tau)$ is a
modular form of weight $12$ on $\Gamma_1$ which does not vanish on
$\mathfrak{H}$ and has $\mathrm{ord}_q\Delta(\tau)=1$
(\cite{Silverman} Chapter I $\S$3). Hence the identity (\ref{cusp})
and the inequality (\ref{inequality}) imply that $g$ is
well-defined.
\par
If $\phi$ and $\phi'$ are two distinct elements of $J_k^m$, then
there exists a point $\tau_0\in\mathfrak{H}$ such that the function
$(\phi-\phi')(\tau_0,z)$ is not identically zero and has no zeros on
the boundary of the fundamental domain generated by $\tau_0$ and $1$
(for the torus $\mathbb{C}/[\tau_0,1]$). Suppose that $\phi$ and
$\phi'$ have the same image via $g$. Then for every $j$
\begin{equation*}
0=(\phi_{X_j}\Delta^m-\phi'_{X_j}\Delta^m)(\tau_0)
=(\phi_{X_j}-\phi'_{X_j})(\tau_0)\Delta(\tau_0)^m.
\end{equation*}
Since $\Delta(\tau_0)\neq0$, we get
\begin{eqnarray*}
0&=&(\phi_{X_j}-\phi'_{X_j})(\tau_0)=(\phi|''_k X_j-\phi'|''_k
X_j)(\tau_0,0)\\
&=&\bigg(\psi(\begin{matrix}u_j&v_j\end{matrix})
\exp(\eta(u_j\tau_0+v_j,[\tau_0,1])\tfrac{1}{2}(u_j\tau_0+v_j))\bigg)^k
(\phi-\phi')(\tau_0,u_j\tau_0+v_j).
\end{eqnarray*}
This implies $(\phi-\phi')(\tau_0,u_j\tau_0+v_j)=0$. Thus the
function $(\phi-\phi')(\tau_0,z)$ has at least $-k+1$ distinct zeros
in the fundamental domain $F$ generated by $\tau_0$ and $1$. (When
$(\phi-\phi')(\tau_0,z)$ has zeros on the boundary $\partial F$, we
slightly move $F$ into a new domain $F'$ so that the points
$u_j\tau_0+v_j$ still lie inside $F'$ and the boundary $\partial F'$
has no zeros of $(\phi-\phi')(\tau_0,z)$.) But, this contradicts
Proposition \ref{zero}. Therefore, $g$ is injective; hence we obtain
\begin{equation*}
\mathrm{dim}_\mathbb{C}J_k^m\leq\sum_{j=1}^{-k+1}\mathrm{dim}_\mathbb{C}
M_{k+12m}(\Gamma_j)<\infty
\end{equation*}
as desired.
\end{proof}

\bibliographystyle{amsplain}

\end{document}